%% file: approachable.tex
\documentclass{amsart}

\input{artcfg}

\DeclareMathOperator{\CAT}{CAT}
\newcommand{\emb}{\mathfrak{e}}

\newcommand{\tpcX}{\cl{X} \times \dbX}
\newcommand{\Par}{\mathcal{P}}
\newcommand{\PPar}{P}
\newcommand{\CS}{\mathcal{CS}}

\DeclareMathOperator{\Closedsets}{\mathfrak{C}}

\newcommand{\BX}{\partial X}

\renewcommand{\emptyset}{\varnothing}

\newcommand{\dX}{d_X}

\newcommand{\doubleres}{\overline}

\newcommand{\ra}{\to}
\newcommand{\sbs}{\subset}
\newcommand{\p}{\partial}
\newcommand{\0}[1]{_{_{#1}}}

\newcommand{\bS}{\Bbb{S}}
\newcommand{\barS}{{\bar{S}}}
\newcommand{\bg}{{\bar{g}}}

\title{Intrinsic Rank in CAT(0) Spaces}

\author{Pedro Ontaneda}
\address{Binghamton University, Binghamton, New York, USA}
\email{pedro@math.binghamton.edu}

\author{Russell Ricks}
\email{rricks@binghamton.edu}

\begin{document}

\begin{abstract}
Let $X$ be a proper, geodesically complete $\CAT(0)$ space which satisfies Chen and Eberlein's duality condition.
We show the existence of a strong notion of rank for $X$ by proving that the parallel sets $P_v$ of geodesics $v$ in $X$ are generically flat.
More precisely, let $GX$ be the space of parametrized unit-speed geodesics in $X$.
There is a unique $k$ and a dense $G_\delta$ set $\mathcal{A}$ in $GX$ such that $P_v$ is isometric to flat Euclidean space $\mathbb{R}^k$, for all $v \in \mathcal{A}$.
It follows that $\mathbb{R}^k$ isometrically embeds in $P_v$ for every $v \in GX$.
\end{abstract}

\maketitle

\section{Introduction}

Let $M$ be a Hadamard manifold, that is, a complete simply connected Riemannian manifold of nonpositive curvature.
We will denote the set of (unit-speed parametrized) geodesic lines in $M$ by $GM$.
For a geodesic line $v$ in $M$ the \defn{parallel set} $P_v$  of $v$ is the subset of $M$ formed by the union of the images of all geodesic lines parallel to $v$.
The set $P_v$ is a closed convex subset of $M$.
Recall that the \defn{rank} of a geodesic line $v$ is the number of linearly independent parallel Jacobi fields along $v$, and the \defn{rank of} $M$ is the minimum rank of all geodesic lines in $M$.

If $P_v$ is isometric to some Euclidean space $\mathbb{R}^k$, then clearly the rank of $v$ is at least $k$.
The converse is true generically, under a condition called the \defn{duality condition} (see remark after theorem below).
We state the result explicitly.

\begin{introtheorem} [Theorem 2.2 of \cite{eh90}]	\label{open dense}
Let $M$ be a Hadamard manifold.
Assume that $M$ satisfies the duality condition.
If the rank of $M$ is $k$, then there is an open dense set $\mathcal{A} \subseteq GM$ such that $P_v$ is isometric to $\mathbb{R}^k$, for all $v\in \mathcal{A}$.
Moreover, Euclidean space $\mathbb{R}^k$ isometrically embeds in $P_v$ for every $v \in GX$.
\end{introtheorem}

\begin{remark}
The duality condition is equivalent to the following property: the set of $\Gamma$-recurrent geodesics is dense, for some (not necessarily discrete) subgroup $\Gamma$ of the isometry group of $M$.
If $M$ admits a geometric action---that is, properly discontinuous, cocompact group action by isometries---then it satisfies the duality condition \cite[p.39]{ballmann} due to the presence of the Liouville measure on $GM$.
However, if $M$ is homogeneous then $M$ satisfies the duality condition only if $M$ is a symmetric space \cite[Proposition 4.9]{eb82}.
\end{remark}

In this paper, we are concerned with generalizing the result above to proper, geodesically complete $\CAT(0)$ spaces.
Our first comment in this regard is that \thmref{open dense} is not true in the $\CAT(0)$ setting, even if the underlying space is a manifold.
In fact, one can construct a proper, geodesically complete $\CAT(0)$ metric on $\mathbb{R}^2$ under a rank one geometric action (which satisfies the duality condition by \cite[Proposition 3.6]{ricks-mixing}), but the set of geodesics $v$ with $P_v$ isometric to $\mathbb{R}$ does not contain an open dense set.
Of course this metric is not Riemannian.
We give the idea of this construction in the following paragraph.

Take a surface $S$ of genus $\geq 2$, with a hyperbolic metric.
Enumerate the closed geodesics $c_1, c_2,...$ and add thinner and thinner cylinders around each $c_i$.
If we do this with some care (see the Appendix for a few more details), we obtain a well defined nonpositively curved geodesic metric; we denote the surface with this new metric by $S'$.
Each $c_i$ will be homotopic to a closed geodesic $c_i'$ in $S'$.
Moreover each $c_i'$ is contained, by construction, in a thin cylinder.
Now just take the universal cover $X$ of $S'$, which is homeomorphic to $\mathbb{R}^2$.
Let $B$ be the set of all possible liftings of the $c_i'$ (with varying base points).
Then $B$ is dense in $GX$ and for every $v\in B$, $P_v$ is an infinite strip, hence not isometric to any Euclidean space.
Therefore the set $\mathcal{A}$ of all $v\in GX$ with $P_v$ isometric to $\mathbb{R}$ does not contain an open set.

The metric in the example above is not smoothly Riemannian, but it is $C^0$-Riemannian, and it is the limit of smooth Riemannian metrics.

However, in the example above one can prove that the set $\mathcal{A}$ is a dense $G_\delta$ set.
Recall that a dense $G_\delta$ set is a countable intersection of dense open sets.
These sets behave like open dense sets in the sense that the intersection of two dense $G_\delta$ sets is also a dense $G_\delta$ set.
Moreover, any countable intersection of dense $G_\delta$ sets is also a dense $G_\delta$ set.
This is why a property is called generic if the set of objects satisfying the property is a dense $G_\delta$ set.
The main result of this paper is a generalization of \thmref{open dense}.

\begin{maintheorem}				\label{intrinsic rank}
Let $X$ be a proper, geodesically complete $\CAT(0)$ space.
Assume that $X$ satisfies the duality condition.
Then there is a unique $k$ and a dense $G_\delta$ set $\mathcal{A} \subseteq GX$ such that $P_v$ is isometric to $\mathbb{R}^k$, for all $v \in \mathcal{A}$.
Moreover, Euclidean space $\mathbb{R}^k$ isometrically embeds in $P_v$ for every $v \in GX$.
\end{maintheorem}

Note that now, assuming the duality condition, we can \emph{define} the rank of $X$ as the number $k$ given by the theorem.
We call this number the \emph{intrinsic rank} of $X$.

We remark that a result of Ballmann (Theorem III.2.4 in \cite{ballmann}) states the following for a proper, geodesically complete $\CAT(0)$ space that satisfies the duality condition:  If the geodesic flow on $GX$ does not have a dense orbit mod $\Isom(X)$---equivalently, the $\Isom(X)$-action on $\bd X$ is not minimal---then the intrinsic rank of $X$ is at least $2$.
(In other words, every geodesic in $X$ is contained in a full $2$-flat.)
Our Main Theorem not only applies even if the boundary action is not known to be non-miminal, but it also provides a dense set of geodesics for which the parallel set is exactly a $k$-flat instead of some larger space.

It is interesting to note that one can propose several natural ways of generalizing the concept of Riemannian rank to geodesic spaces.
For instance, the rank of a geodesic $v$ is two if it is contained in a flat plane, or it is contained in the interior of an infinite strip, or if it bounds a half plane, or half bounds an infinite strip.
One can also fix the width of the strip, or take strips with varying widths.
In higher dimensions there are even more choices: $v$ is contained in an Euclidean space $\mathbb{R}^k$, or $v$ is contained in the boundary of a half Euclidean space, or in the boundary of a quarter Euclidean space (which could be ``at the vertex" or not), and so on.
Our result shows that, under the duality condition, all these possible generalizations are equivalent; they yield the same concept of the rank of a space.

In particular, if one can find a single geodesic line $v$ in $X$ such that $P_v$ does not contain a flat plane, then $X$ has rank one and there is a dense $G_\delta$ set of geodesics $w$ with $w$ being parallel only to itself.
It follows (still assuming the duality condition, of course) that $X$ contains a dense set of so-called rank one axes---geodesics on which nontrivial elements of $\Isom(X)$ act by translation and which do not bound a flat half-plane \cite[Theorem~ III.3.4]{ballmann}.
Under such conditions $X$ is known to exhibit a fair degree of hyperbolic behavior.

This generalization of the concept of rank is also relevant in the formulation of a $\CAT(0)$ version of the Rank Rigidity Theorem of Ballmann, Brin, Burns, Eberlein, Heber, and Spatzier for Hadamard manifolds \cite{ballmann,burns-spatzier,eh90}.
This celebrated result states that if a Hadamard manifold satisfies the duality condition and has rank at least two, then it is either a symmetric space or a Riemannian product.
A version of this result holds for $\CAT(0)$ cube complexes, by Caprace and Sageev \cite{caprace-sageev}.
It is conjectured that an appropriate generalization also holds for $\CAT(0)$ spaces.
Using the intrinsic rank we can state this conjecture in the following way.

\begin{crrconjecture}				\label{rank rigidity}
Let $X$ be a proper, geodesically complete $\CAT(0)$ space that satisfies the duality condition.
If the intrinsic rank of $X$ is at least two, then $X$ is either a product, or a symmetric space or Euclidean building.
\end{crrconjecture}

In \secref{sec:application}, we prove (\corref{maximal rank corollary}) a weak version of the $\CAT(0)$ Rank Rigidity Conjecture, subject to a dimension restriction.

\begin{maincorollary}				\label{main corollary}
Let $X$ be a proper, geodesically complete $\CAT(0)$ space that satisfies the duality condition.
Let $X = X_1 \times \dotsb \times X_n$ be the maximal de Rham decomposition of $X$, so that each $X_i$ is neither compact nor a product.
Assume $\rank(X) = 1 + \dim(\bdT X)$.
Then for each de Rham factor $X_i$ of $X$, either
\begin{enumerate}
\item\label{it:intro mrc minimal}  $\Isom(X_i)$ acts minimally on $\bd X_i$ and the geodesic flow on $GX_i$ has a dense orbit mod $\Isom(X_i)$, or
\item\label{it:intro mrc building}  $X_i$ is a symmetric space or Euclidean building of rank at least two.
\end{enumerate}
\end{maincorollary}

We remark that Ballmann \cite[p.7]{ballmann} identified three problems to solve in extending the Rank Rigidity Theorem to proper, geodesically complete $\CAT(0)$ spaces.
The first of these problems was to define the rank of a $\CAT(0)$ space (that satisfies the duality condition) in such a way that rank $k \ge 2$ if and only if every geodesic is contained in a Euclidean $k$-flat (a subspace of $X$ isometric to flat Euclidean $\R^k$), and rank one implies hyperbolic behavior.
The second was to show that if every geodesic of $X$ was contained in a $k$-flat, $k \ge 2$ (and $X$ satisfies the duality condition), then $X$ is a product, a symmetric space, or a Euclidean building.
The third was to show that if the space admits a properly discontinuous, cocompact group action by isometries, then it satisfies the duality condition.
Our Main Theorem solves the first problem (and even provides a dense $G_\delta$ set of geodesics whose parallel set is precisely a $k$-flat); \corref{maximal rank corollary} gives a partial solution to the second; the third is still open.

We make a final observation about the duality condition.
This is equivalent to requiring the geodesic flow to be nonwandering.
It is thus a very natural hypothesis from a dynamical perspective, and forms an essential ingredient in the proof of all the Rank Rigidity results for Riemannian manifolds mentioned previously.
In fact, the earliest results assumed more---compactness or finite volume, both of which imply the geodesic flow is nonwandering (the proof going back to Poincar\'e).
The most general version of Rank Rigidity for manifolds (due to Eberlein and Heber) assumes only the duality condition.
The exception here is telling:  Only by relying on the very strong combinatorial structure of the hyperplanes in CAT(0) cube complexes were Caprace and Monod able to elide the duality condition in their proof of Rank Rigidity.
Thus, since the duality condition is the essential hypothesis for proving Rank Rigidity for manifolds,
this same dynamical information is the appropriate choice to prove
Rank Rigidity for general CAT(0) spaces.

\section{Preliminaries}

Let $X$ be a metric space.
A \defn{geodesic} in $X$ is an isometric embedding $v \colon \R \to X$, a \defn{geodesic ray} is an isometric embedding $\alpha \colon [0,\infty) \to X$, and a \defn{geodesic segment} is an isometric embedding $\sigma \colon [0,r] \to X$ for some $r > 0$.
The space $X$ is called \defn{geodesic} if every pair of distinct points is connected by a geodesic segment; $X$ is \defn{uniquely geodesic} if the segment is always unique (up to reversing parametrization).
Also, $X$ is called \defn{geodesically complete} if every geodesic segment in $X$ extends to a full (not necessarily unique) geodesic in $X$.

We write $GX$ for the set of all geodesics in $X$, endowed with the compact-open topology (i.e.~ the topology of uniform convergence on compact subsets).
This space is completely metrizable when $X$ is complete.
There is also a canonical \defn{geodesic flow} $g^t$ on $GX$ given by $(g^t v)(s) = v(s + t)$.

A uniquely geodesic metric space $X$ is called \defn{CAT(0)} if, for every triple of distinct points $x,y,z \in X$, the geodesic triangle $\triangle(x,y,z) \subset X$ is no fatter than the corresponding comparison triangle $\overline{\triangle}(x,y,z)$ in Euclidean $\R^2$ (the triangle with the same edge lengths).
For more on $\CAT(0)$ spaces, see \cite{ballmann} or \cite{bridson}.

Let $X$ be a complete $\CAT(0)$ space.
The \defn{visual boundary} (written $\bd X$) of $X$ is the set of equivalence classes of asymptotic geodesic rays.
Equivalently, one can fix a basepoint $x_0 \in X$ and take all geodesic rays emanating from $x_0$.
The standard topology on $\bd X$ is the compact-open topology, often called the \defn{cone} or \defn{visual} topology.
(This topology does \emph{not} depend on choice of basepoint.)

Viewing each point $x \in X$ as a geodesic segment from a fixed basepoint to $x$, the topology of uniform convergence on compact subsets naturally gives a topology on $\overline{X} = X \cup \bd X$.
If $X$ is proper (meaning all closed balls are compact), then both $\bd X$ and $\overline{X} = X \cup \bd X$ are compact metrizable spaces.
For each $v \in GX$, we will write $v(\infty) = \lim_{t \to +\infty} v(t) \in \bd X$ and $v(-\infty) = \lim_{t \to -\infty} v(t) \in \bd X$.

We now define parallel sets and cross sections; a version of each exists both in $X$ and $GX$.
Let $v \in GX$.
A geodesic $w \in GX$ is \defn{parallel} to $v$ if the map $t \mapsto d(v(t),w(t))$ is constant.
Let $\Par_v \subset GX$ be the set of geodesics parallel to $v$, and let $P_v$ be the set of points in $X$ that lie on some $w \in \Par_v$.
We call $\Par_v$ the \defn{parallel set} of $v$ in $GX$, and $P_v$ the \defn{parallel set} of $v$ in $X$.
It is a standard fact of CAT(0) geometry that $P_v$ is a convex subset of $X$, isometric to $C_v \times \R$, where $C_v$ is a closed convex subset of $P_v$ containing $v(0)$.  We call $C_v$ the \defn{cross section} of $P_v$ in $X$, or just the \defn{cross section} of $v$ in $X$.
We call the set $\CS_v = \setp{w \in \Par_v}{w(0) \in C_v}$ the \defn{cross section} of $v$ in $GX$.
Notice that footpoint projection $GX \to X$ (given by $w \mapsto w(0)$) bijectively carries $\Par_v$ to $P_v$ and $\CS_v$ to $C_v$.

There is a simple and useful metric on $GX$, defined by
\[d(v, w) = \sup_{t \in \R} e^{-\abs{t}} d(v(t), w(t)) \qquad \text{for all } v,w \in GX.\]
This metric is complete if $X$ is complete, and proper if $X$ is proper (by the Arzel\`a-Ascoli theorem).
It is also isometry-invariant and induces the topology of uniform convergence on compact subsets.
Moreover, within each parallel set $\Par_v$, it is flow-invariant and therefore restricts to the metric on $P_v$.
Thus footpoint projection restricts to an isometry $\Par_v \to P_v$ and $\CS_v \to C_v$ for each $v \in GX$.

\section{Duality and Recurrence}

We first describe the duality condition, introduced by Chen and Eberlein \cite{chen-eberlein} for Hadamard manifolds.
The results in this section should be familiar to the experts.

Let $X$ be a complete $\CAT(0)$ space.
Write $\Isom(X)$ for the isometry group of $X$.
For a subgroup $\Gamma \le \Isom(X)$, two points $\xi, \eta \in \bd X$ are called \defn{$\Gamma$-dual} if there exists a sequence $(\gamma_n)$ in $\Gamma$ such that $\gamma_n x \to \xi$ and $\gamma_n^{-1} x \to \eta$ for some (hence any) $x \in X$.
The subgroup $\Gamma$ is said to satisfy the \defn{duality condition} if $v(\infty)$ and $v(-\infty)$ are $\Gamma$-dual for every $v \in GX$.
We will say that $X$ satisfies the duality condition if $\Isom(X)$ does.

Notice $X$ satisfies the duality condition whenever any subgroup $\Gamma$ of $\Isom(X)$ does.
In fact, if the duality condition holds for a group $\Gamma$, then it holds not only for arbitrary supergroups but also for finite-index subgroups.
Moreover, when $X$ is proper and geodesically complete, if $X$ splits as a $\CAT(0)$ product $X = X_1 \times X_2$ and satisfies the duality condition, then $X_1$ and $X_2$ also satisfy the duality condition.
These facts are proved in \cite[Remark III.1.10]{ballmann}.

The reason the duality condition is of interest dynamically is its relationship to recurrence and nonwandering, which we describe below.
A geodesic $v \in GX$ is called (forward) \defn{$\Gamma$-recurrent} if there exist sequences $t_n \to +\infty$ and $\gamma_n \in \Gamma$ such that $\gamma_n g^{t_n} (v) \to v$ as $n \to \infty$; it is called \defn{$\Gamma$-nonwandering} if there exists sequences $v_n \in GX$, $t_n \to +\infty$, and $\gamma_n \in \Gamma$ such that $v_n \to v$ and $\gamma_n g^{t_n} (v_n) \to v$ as $n \to \infty$.
(Notice $\Gamma$-recurrent implies $\Gamma$-nonwandering.)

These notions are related to the usual notions of recurrence and nonwandering as follows.
If $\Gamma \le \Isom(X)$ is discrete, then $v \in GX$ is $\Gamma$-recurrent (respectively, $\Gamma$-nonwandering) if and only if its projection onto $\modG{GX}$ is recurrent (respectively, nonwandering) under the geodesic flow $g^t_\Gamma$ on $\modG{GX}$.
The situation with nonwandering is completely similar.

\begin{definition}
We will say $v \in GX$ is \defn{recurrent} if $v$ is $\Isom(X)$-recurrent, and \defn{nonwandering} if $v$ is $\Isom(X)$-nonwandering.
Note this departure from standard usage creates minimal potential for confusion because no $v \in GX$ can ever be $\left< \id \right>$-recurrent or $\left< \id \right>$-nonwandering in a $\CAT(0)$ space.
\end{definition}

The relationship with duality is derived from the following result, originally due to Eberlein \cite{eb72} in the case of nonpositively curved smooth Riemannian manifolds.

\begin{lemma} [Lemma III.1.1 of \cite{ballmann}]	\label{Eberlein}
Let $X$ be a geodesically complete $\CAT(0)$ space, and let $\Gamma$ be a subgroup of $\Isom(X)$.
If $v,w \in GX$ are such that $v(\infty)$ and $w(-\infty)$ are $\Gamma$-dual, then there exist $\gamma_n \in \Gamma$, $t_n \to +\infty$, and $v_n \in GX$ such that $v_n \to v$ and $\gamma_n g^{t_n} v_n \to w$.
\end{lemma}

In particular (see the discussion preceding Corollary III.1.4 in \cite{ballmann}):

\begin{corollary}
Let $X$ be a $\CAT(0)$ space.
If every $v \in GX$ is nonwandering, then $X$ satisfies the duality condition.
The converse holds if $X$ is geodesically complete.
\end{corollary}

We will want two results later, which we record here.
The first is the following standard result (see \cite[Corollary III.1.5]{ballmann}, for instance, for a proof).

\begin{lemma}					\label{forward recurrents dense}
Let $X$ be a complete $\CAT(0)$ space.
If every $v \in GX$ is nonwandering, then the recurrent geodesics form a dense $G_\delta$ set in $GX$.
\end{lemma}

\begin{proof}
For each $k,n \in \N$, define $U_{k,n} := \bigcup_{t \ge n} \bigcup_{\gamma \in \Gamma} \setp{v \in GX}{d(v, \gamma g^t v) < \frac{1}{k}}$.
Each $U_{k,n}$ is open and dense by the nonwandering hypothesis, hence $\bigcap_{k,n \in \N} U_{k_n}$ is dense $G_\delta$ in $GX$ because $GX$ is complete.
But $\bigcap_{k,n \in \N} U_{k_n}$ is precisely the set of recurrent geodesics.
Thus the conclusion of the lemma holds.
\end{proof}

The next lemma follows easily from Lemma 6.7 and Lemma 6.10 in \cite{ricks-mixing}, but we provide a proof here for the convenience of the reader.

\begin{lemma}					\label{recurrent embeddings}
Let $X$ be a proper $\CAT(0)$ space.
If $v \in GX$ is recurrent, then for every $w \in GX$ such that $w(\infty) = v(\infty)$, there is an isometric embedding $(\CS_w, w) \into (\CS_v, v')$ for some $v' \parallel v$ such that $d(v,v') \le d(v,w)$.
\end{lemma}

(Here we have adopted the notation $(\CS_w, w) \into (\CS_v, v')$ to mean $\CS_w \to \CS_v$ is an isometric embedding that sends $w \mapsto v'$.)

\begin{proof}
Let $w \in GX$ with $w(\infty) = v(\infty)$.
Since $v$ is recurrent, there exist $\gamma_n \in \Isom(X)$ and increasing $t_n \to +\infty$ such that $\gamma_n g^{t_n}(v) \to v$.
Since $w$ and $v$ are forward asymptotic, $d(\gamma_n g^{t_n}(v), \gamma_n g^{t_n}(w))$ is nonincreasing.
By Arzel\`a-Ascoli, we may pass to a subsequence for which the isometric embeddings $\gamma_n g^{t_n} \res{\CS_w} \colon \CS_w \into GX$ converge to an isometric embedding $\psi \colon (\CS_w, w) \into (\CS_u, u)$ for some $u \parallel v$.
The desired isometric embedding is now $g^r \circ \psi \colon (\CS_w, w) \into (\CS_v, g^r u)$, where $r \in \R$ is chosen so that $g^r u \in \CS_v$.
\end{proof}

A variation on \lemref{recurrent embeddings} is the following property of recurrent geodesics, which is due to Guralnik and Swenson \cite[Corollary 3.24]{gs} in the case that the sequence $(\gamma_n)$ lies completely in some discrete subgroup of $\Isom(X)$.

\begin{lemma}					\label{recurrent suspensions}
Let $X$ be a proper $\CAT(0)$ space.
Let $w \in GX$ be recurrent, so there exist $\gamma_n \in \Isom(X)$ and $t_n \to +\infty$ such that $\gamma_n g^{t_n} (w) \to w$.
Then for every $p \in \bd X$, every accumulation point of $(\gamma_n p)$ lies in $\bd \PPar_w$.
\end{lemma}

\begin{proof}
Let $q \in \bd X$ be an accumulation point of $(\gamma_n p)$.
Passing to a subsequence, we may assume $\gamma_n p \to q$.
Note that $\gamma_n (w(\infty)) \to w(\infty)$.
Now put $x = w(0)$, and for each $n$, put $x_n = w(t_n) = g^{t_n} w(0)$.
Then
\[\angle_{x} (q, w(\infty))
\ge \limsup \angle_{\gamma_n x_n} (\gamma_n p, \gamma_n w(\infty))
= \limsup \angle_{x_n} (p, w(\infty))
= \angle (p, w(\infty))\]
by standard $\CAT(0)$ geometry.
But by lower semicontinuity of $\angle$, we know
\[\angle (p, w(\infty))
= \angle (\gamma_n p, \gamma_n w(\infty))
\ge \angle (q, w(\infty))
\ge \angle_x (q, w(\infty)).\]
Therefore, $\angle_x (q, w(\infty)) = \angle (p, w(\infty))$.
Papasoglu and Swenson's $\pi$-convergence theorem (stated for a discrete group of isometries in \cite[Lemma 19]{ps}, but the proof does not use this assumption)
then shows
\[\angle_x (q, w(-\infty)) \le \angle (q, w(-\infty)) \le \pi - \angle (p, w(\infty)) = \pi - \angle_x (q, w(\infty)).\]
Since $\angle_x (w(-\infty), w(\infty)) = \pi$, we see that either $q = w(\pm \infty)$ or $q$ lies in the ideal boundary of a flat half-plane bounded by $w$, i.e.~ in either case, $q \in \bd \PPar_w$.
\end{proof}

\section{Complete approachability}

Call $v \in GX$ \defn{completely approachable} if, for every $x \in C_v$ and sequence $v_n \to v$ in $GX$, there exists $x_n \to x$ in $X$ such that each $x_n \in C_{v_n}$.
This terminology reflects the idea that every $x \in C_v$ is approachable by $x_n \in C_{v_n}$ for every $v_n \to v$.

For a metric space $Z$, write $\Closedsets(Z)$ for the space of closed subsets of $Z$, with the Hausdorff topology.

\begin{lemma}					\label{ca equivalences}
Let $X$ be a proper $\CAT(0)$ space, and let $v \in GX$.
The following are equivalent.
\begin{enumerate}
\item\label{it: ca}
$v$ is completely approachable.
\item\label{it: seq in X}
For every $x \in C_v$ and sequence $v_n \to v$ in $GX$, there exist $x_n \to x$ in $X$ such that each $x_n \in C_{v_n}$.
\item\label{it: seq in GX}
For every $w \in \CS_v$ and sequence $v_n \to v$ in $GX$, there exist $w_n \to w$ in $GX$ such that each $w_n \in \CS_{v_n}$.
\item\label{it: pt of cont}
The extended cross-section map $\cl{\CS} \colon GX \to \Closedsets(\cl{GX})$ is continuous at $v$, where $\cl{GX}$ is one-point compactification of $GX$ and $\cl{\CS}(w) := \CS_w \cup \set{\infty}$. \end{enumerate}
\end{lemma}

\begin{proof}
\itemrefstar{it: seq in X} is the definition of \itemrefstar{it: ca}.
The equivalence of \itemrefstar{it: seq in GX}$\iff$\itemrefstar{it: seq in GX} is trivial, and \itemrefstar{it: seq in GX}$\iff$\itemrefstar{it: pt of cont} because $\CS$ is upper semicontinuous.
\end{proof}

It is a standard fact (see, for example, \cite[Theorem A.1]{akin97}) that every upper semicontinuous map $Y \to \mathcal{C}(Z)$, where $Y$ is a complete metric space and $Z$ a compact metric space, has a dense $G_\delta$ set of continuity points.
Thus we have the following.

\begin{lemma}					\label{completely approachables}
The completely approachable geodesics form a dense $G_\delta$ set in $GX$.
\end{lemma}

\begin{corollary}				\label{regular dense}
Assume every geodesic $v \in GX$ is nonwandering.
The set $\mathcal{U} \subseteq GX$ of geodesics that are both completely approachable and recurrent is dense $G_\delta$ in $GX$.
\end{corollary}

We will denote the set of completely approachable geodesics by $\mathcal{A}$.

\begin{lemma}					\label{Eberlein plus approachable}
Let $v \in \mathcal{A}$ and $w \in GX$.
If $v(\infty)$ and $w(-\infty)$ are $\Isom(X)$-dual, then there is an isometric embedding $(\CS_v, v) \into (\CS_w, w)$.
\end{lemma}

\begin{proof}
By \lemref{Eberlein}, there exist $\gamma_n \in \Gamma$, $t_n \to +\infty$, and $v_n \in GX$ such that $v_n \to v$ and $\gamma_n g^{t_n} v_n \to w$.
Because $v$ is completely approachable, $\CS_{v_n} \!\!\to \CS_v$.
By upper semicontinuity of $\CS$ we find
$\CS_w \supseteq \lim \gamma_n g^{t_n} \CS_{v_n}$.
Since each $\gamma_n g^{t_n}$ is an isometry on $GX$ which preserves cross sections, the limit of isometries $\varphi_n = \gamma_n g^{t_n} \res{\CS_{v_n}}$ is the desired isometric embedding
$(\CS_v, v) \into (\CS_w, w)$.
\end{proof}

\begin{corollary}				\label{approachable plus nonwandering}
If $v \in GX$ is completely approachable and nonwandering then for every $w \in GX$ with $w(\infty) = v(\infty)$, there is an isometric embedding $(\Par_v, v) \into (\Par_w, w)$.
\end{corollary}

\begin{lemma}					\label{reciprocal embeddings}
Let $Y$ and $Z$ be proper metric spaces, and let $y_0 \in Y$ and $z_0 \in Z$.
If both $f \colon (Y,y_0) \to (Z,z_0)$ and $g \colon (Z,z_0) \to (Y,y_0)$ are isometric embeddings, then both $f$ and $g$ are isometries.
\end{lemma}

\begin{remark}
It may be that $g \neq f^{-1}$.
\end{remark}

\begin{proof}
The composition $g \circ f$ is an isometry on each closed metric ball $\cl{B}_Y(y_0,R)$ in $Y$ by compactness \cite[Theorem 1.6.14]{burago}, hence $g$ is surjective.
Similarly for $f$.
\end{proof}

\begin{corollary}				\label{canonical parallel set on horosphere}
Assume every geodesic $v \in GX$ is nonwandering.
If $v,w \in \mathcal{A}$ and $v(\infty) = w(\infty)$, then there is an isometry $(\Par_v, v) \to (\Par_w, w)$.
\end{corollary}

\section{Lifting convergent sequences}

A map $f \colon X \to Y$ is called \defn{open} at $x \in X$ if, for every open neighborhood $U$ of $x$ in $X$, the image $f(U)$ contains an open neighborhood of $f(x)$ in $Y$.

\begin{lemma}					\label{open-map equivalences}
Let $f \colon X \to Y$ be a continuous map between metric spaces, and let $x$ be a point of $X$.
The following are equivalent.
\begin{enumerate}
\item  $f$ is open at $x$.
\item  For every open neighborhood $U$ of $x$ in $X$, the image $f(U)$ contains an open neighborhood of $f(x)$ in $Y$.
\item  For every sequence $y_n \to f(x)$ in $Y$, there are subsequences $n_k$, and $x_k \in X$ with $f(x_k) = y_{n_k}$, and $x_k \to x$.
\item  For every sequence $y_n \to f(x)$ in $Y$, there exists a sequence $x_n \to x$ in $X$ such that $f(x_n) = y_n$ for all $n$.
\end{enumerate}
\end{lemma}

\begin{proof}
The proof is straightforward and left as an exercise to the reader.
\end{proof}

Our interest in open maps comes from \thmref{conditional Fubini}, which allows us to take dense $G_\delta$ slices of dense $G_\delta$ sets.
The following lemma is the dense open set version.

\begin{lemma}
Let $X$ and $Y$ be complete metric spaces with $X$ separable.
Let $C \subseteq X \times Y$ and let $U \subseteq C$ be relatively open and dense in $C$.
Let $\pi_Y \colon X \times Y \to Y$ be coordinate projection onto $Y$.
Assume the restriction $\doubleres{\pi_Y} \colon C \to \pi_Y (C)$ is open at every point of $U$.
Then
\begin{equation*}
Y_U^C := \setp{y \in Y}{(X \times \set{y}) \cap U \textnormal{ is (open and) dense in } (X \times \set{y}) \cap C}
\end{equation*}
contains a dense $G_\delta$ subset of $Y$.
\end{lemma}

\begin{proof}
Let $(V_n)$ be a countable basis for $X$.
Let
\[E_n = \setp{y \in Y}{(V_n \times \set{y}) \cap U = \emptyset \text{ but }
(V_n \times \set{y}) \cap C \neq \emptyset}.\]
Then $Y \setminus Y_U^C = \bigcup_n E_n$.

Now let $C_n = (V_n \times Y) \cap C$ and $U_n = (V_n \times Y) \cap U$.
Notice that $E_n = f_n (C_n) \setminus f_n (U_n)$, where $f_n = \doubleres{\pi_Y} \res{C_n}$ is the restriction of $\doubleres{\pi_Y}$ to $C_n \to \pi_Y (C)$.
Since $V_n \times Y$ is open, $U_n$ is relatively open and dense in $C_n$; moreover, $f_n$ is continuous and open at every point of $U$.
Thus $f_n (U_n)$ is relatively open and dense in $f_n(C_n)$.
It follows that $E_n$ is relatively nowhere dense in $f_n(C_n)$, hence nowhere dense in $Y$.
Thus $Y \setminus Y_U^C$ is the countable union of nowhere dense sets.
\end{proof}

Taking countable intersections we obtain the following.

\begin{theorem}					\label{conditional Fubini}
Let $X$ and $Y$ be complete metric spaces with $X$ separable.
Let $C \subseteq X \times Y$ and let $A \subseteq C$ contain a subset which is dense $G_\delta$ in $C$.
Let $\pi_Y \colon X \times Y \to Y$ be coordinate projection onto $Y$.
Assume the restriction $\doubleres{\pi_Y} \colon C \to \pi_Y (C)$ is open at every point of $C$.
Then
\begin{equation*}
Y_A^C := \setp{y \in Y}{(X \times \set{y}) \cap A \textnormal{ contains a dense $G_\delta$ subset of } (X \times \set{y}) \cap C}
\end{equation*}
contains a dense $G_\delta$ subset of $Y$.
\end{theorem}

\begin{remark}
The slice $(X \times \set{y}) \cap A$ may be empty for $y \in Y_A^C$, but only if $(X \times \set{y}) \cap C$ is empty.
Thus for many applications one must show that $(X \times \set{y}) \cap C$ is not empty for some dense $G_\delta$ set of $y \in Y$, and then one finds that the set
\begin{align*}
\hat Y_A^C := \{y \in Y \mid (X \times \set{y}) \cap A \textnormal{ is} & \textnormal{ nonempty and contains}  \\ &\textnormal{a dense } G_\delta \textnormal{ subset of } (X \times \set{y}) \cap C\}
\end{align*}
contains a dense $G_\delta$ subset of $Y$.
\end{remark}

Our application for \thmref{conditional Fubini} is to actually to the forward-endpoint map $GX \to \BX$ taking $v \mapsto v(\infty)$.
The topological embedding $\emb \colon GX \to \tpcX$ given by $\emb(v) = (v(0), v(-\infty), v(\infty))$ provides the ambient product structure, and the following lemma shows that the map is open at every point of $GX$.

\begin{lemma}					\label{forward-endpoint map is open}
Let $X$ be a proper, geodesically complete $\CAT(0)$ space.
Let $v \in GX$ and $p=v(\infty)$.
Let $p_n \in \BX$ with $p_n \to p$.
Then there are subsequences $n_k$, and $v_k \in GX$ with $v_k(\infty)=p_{n_k}$, and $v_k \to v$.
\end{lemma}

\begin{proof}
Write $x_k=v(-k)$.
For each $k,n$, choose $w_{k,n}$ such that $w_{k,n}(-k)=x_k$ and $w_{k,n}(\infty)=p_n$.
So for each fixed $k$, the geodesics $w_{k,n}$ keep $w_{k,n}(-k)=x_k$ while $w_{k,n}(\infty) \to p$.
Thus we may find for each $k$ some $n_k \ge k$ such that $w_{k,n_k}(0) < 1/k$.
It follows that $v_k := w_{k,n_k} \to v$.
\end{proof}

Recall from \lemref{regular dense} that the set $\mathcal{U} := \setp{v \in \mathcal{A}}{v \text{ is recurrent}}$ is a dense $G_\delta$ subset of $GX$.
For $p \in \BX$, let $GX_p$ be the set of $v \in GX$ such that $v(\infty) = p$.

\begin{corollary}				\label{regular dense slices}
Assume $X$ is geodesically complete and every geodesic $v \in GX$ is nonwandering.
There is a set $b\mathcal{U}$ in $\BX$ that contains a dense $G_\delta$ subset of $\BX$, such that for every $p \in b\mathcal{U}$ the set $\mathcal{U}_p := \mathcal{U} \cap GX_p$ contains a subset that is dense $G_\delta$ in $GX_p$.
\end{corollary}

\begin{proof}
Combine \corref{regular dense}, \lemref{forward-endpoint map is open}, and \thmref{conditional Fubini}.
\end{proof}

\section{Isometric Transitivity}

\begin{lemma}					\label{isometric transitivity}
Let $X$ be a proper, geodesically complete $\CAT(0)$ space that satisfies the duality condition.
Then the isometry group $\Isom(\CS_v)$ is transitive for all $v \in GX$ such that $v(\infty) \in b\mathcal{U}$.
\end{lemma}

\begin{proof}
Let $p \in b\mathcal{U}$.
By \lemref{reciprocal embeddings}, it suffices to construct an isometric embedding $(\CS_v, v) \into (\CS_w, w)$ for all $v \in GX_p$ and $w \parallel v$.
So let $v \in GX_p$ and $w \parallel v$.

By density of $\mathcal{U}_p$ in $GX_p$, there is a sequence $(v_n)$ in $\mathcal{U}_p$ such that $v_n \to v$.
By \lemref{recurrent embeddings} and \corref{approachable plus nonwandering}, for each $n$ we can find isometric embeddings $\varphi_n \colon (\CS_v, v) \into (\CS_{v_n}, v'_n)$, for some $v'_n \parallel v_n$ such that $d(v'_n, v) \le d(v_n, v)$, and $\psi_n \colon (\CS_{v_n}, v_n) \into (\CS_w, w)$.
Thus $\psi_n \circ \varphi_n \colon (\CS_v, v) \into (\CS_w, w'_n)$ is a sequence of isometric embeddings with $w'_n = \psi_n(v'_n) \to w$.
A subsequence of $\psi_n \circ \varphi_n$ converges to an isometric embedding $(\CS_v, v) \into (\CS_w, w)$, as desired.
\end{proof}

A variation on the preceding proof gives us the following corollary.

\begin{corollary}				\label{consistent cross sections}
Let $X$ be a proper, geodesically complete $\CAT(0)$ space that satisfies the duality condition.
Then $\CS_v$ is isometric to $\CS_w$ for all $v,w \in GX$ such that $v(\infty) = w(\infty) \in b\mathcal{U}$.
\end{corollary}

\begin{proof}
Fix $p \in b\mathcal{U}$ and $w \in \mathcal{U}_p$.
Let $v \in GX_p$.
\lemref{recurrent embeddings} gives us an isometric embedding $\varphi \colon (\Par_v, v) \into (\Par_w, w')$ for some $w' \parallel w$.
By \lemref{isometric transitivity}, we may assume $w' = w$.
And \corref{approachable plus nonwandering} gives us an isometric embedding $\psi \colon (\Par_w, w) \into (\Par_v, v)$.
By \lemref{reciprocal embeddings}, $\varphi$ and $\psi$ are isometries.
The corollary follows.
\end{proof}

\section{Intrinsic Rank}

Throughout this section, $X$ is a proper, geodesically complete $\CAT(0)$ space.

\subsection{Parallel sets}

We use the same notation for a geodesic $v\in GX$ and its image $v(\R)$.
Recall that for $v \in GX$, the parallel set $\Par_v \subset GX$ is the set of geodesics parallel to $v$,
and $P_v = \bigcup_{v \in \Par_v} v \subset X$ is isometric to $\Par_v$ under footpoint projection.
If we want to specify the space $X$, we will write $P^X_v$ instead of $P_v$.

Recall that $P_v$ splits isometrically as $P_v = C_v \times v$, where $C_v$ is the cross section of $v$.
Specifically, the isometry $P_v \to C_v \times v$ is given by
$x\mapsto (\pi\0{C_v}(x), \pi_{v}(x))$, where each coordinate is convex projection.
Call these the \defn{$v$-coordinates} of $x \in P_v$.
Sometimes we will write $x=(y,v(a))$ and identify $P_v$ with $C_v \times v$.

We collect some facts about parallel sets.

\begin{enumerate}
\item
Let $w$ be a geodesic in the CAT(0) space $P_v$, with $w$ not parallel to $v$.
There exist a geodesic $u$ in $C_v$ and an angle $\theta\in (0,\pi)$, such that $w(t)=(u(t\sin\theta)\,  ,\,
v(a+t\cos\theta))$ for all $t \in \R$.
Here we have $v(a) = \pi_v(w(0))$.
Note that
$u(t)=\pi\0{C_v}(w(t\csc \theta))$.
Call $u$ the \defn{normalized projection} of $w$ in $C_v$.
Of course $w$ is contained in the 2-flat $u\times v \subseteq P_v$.
(Note: $u \times v$ may not contain $v$, despite the notation.
But it does if $w(0) = v(0)$.)

\item
Let $w$ be a geodesic in $P_v$ which is not parallel to $v$.
Let $u$ be the normalized projection of $w$.
Assume $v(0)=w(0)$.
Since $w$ is contained in the 2-flat $u\times v \subseteq P_v$,
we see that $v$ is contained in $P_w$.
By the same argument, every $v' \parallel v$ such that $v'(0) = w'(0)$ for some $w' \parallel w$ is contained in $P_w = P_{w'}$.

\item
Let $w_1$ and $w_2$ be two geodesics in $P_v$ which are not parallel to $v$.
Let $u_i$ be the normalized projection
of $w_i$.
If $w_1 \parallel w_2$ then $u_1 \parallel u_2$.
This is because the projection $\pi\0{C_v}$ does not increase distances.
\end{enumerate}

\begin{lemma}					\label{pedro 1.2}
Let $w$ be a geodesic in $P_v$ that is not parallel to $v$, such that $v(0)=w(0)$.
Let $u$ be the normalized projection of $w$.
Then, using the identification
$P_v\ra C_v\times v$, we can write
\[P_v\,\cap\,P_w\,=\, P^{C_v}_u\,\times v.\]
\end{lemma}

\begin{remark}
Notice that the right-hand side of the equation above does not
depend directly on $w$, only on the normalized projection $u$ of $w$.
\end{remark}

\begin{proof}
The set $P_v \cap P_w$ is a convex subset of $P_v$ that contains $v$, along with every $v' \parallel v$ such that $v'(0) \in P_w$ (see (ii) above).
Therefore we can write $P_v \cap P_w = E \times v$ for some convex subset $E$ of $C_v$.
Since $w \sbs E \times v$ and $u$ is the normalized projection of $w$,
we see that $u$ is a geodesic in $E$.

We now prove $E \times v \sbs P^{C_v}_u \times v$.
Let $x \in E \times v = P_v \cap P_w$.
Then $x \in w'$, for some $w' \parallel w$.
Since $x\in P_v$ and the distance from $w'$ to $P_v$ is bounded (because $w' \parallel w$ and $w\sbs P_v$), we see that $w'\sbs P_v$.
By (iii) above, $u'$ is parallel to $u$, where $u'$ is the normalized projection of $w'$.
Therefore $x \in u' \times v$, with $u' \parallel u$.
This proves $E \times v \sbs P^{C_v}_u \times v$.
The other inclusion follows from the definitions and the fact that if $u'$ is parallel to $u$ then the 2-flats $u' \times v$ and $u \times v$ are parallel.
\end{proof}

\subsection{The Decomposition Lemma}

We need a lemma about convex sets.

\begin{lemma}					\label{pedro 2.1}
Let $F$ be a closed convex set in the $\CAT(0)$ space $X$, and let $v,w$ be parallel geodesics in $X$ such that $w$ is contained in $F$.
Then $t \mapsto \dX(v(t),F)$ is constant, and there is a geodesic $w'$ in $F$ such that $w' \parallel v$ and $\dX(v,F) = \dX(v,w')$.
\end{lemma}

\begin{proof}
The distance to $F$ is constant because it is a convex and bounded function on $\R$.
Define $w'$ by $w'(t) = \pi_F (v(t))$, where $\pi_F$ is the convex projection $X \to F$.
This geodesic satisfies the desired conditions.
\end{proof}

\begin{lemma} [Decomposition Lemma]		\label{decomposition lemma}
Let $v \in GX$.
Assume that $C_v$ contains a geodesic $u$ with $u(0) = v(0)$.
Further assume there is a sequence $v_n \to v$ in $u \times v \subseteq P_v$  with $v_n \neq v$ such that for every $x \in C_v$ there is a sequence $x_n \to x$ in $X$ where each $x_n \in C_{v_n}$.
Then $C_v = P^{C_v}_u$.
Hence we can write $C_v = E \times u$ for some proper CAT(0) space $E$.
\end{lemma}

\begin{proof}
By reversing the orientation of $u$ if necessary, and passing to a subsequence of $v_n$, we may assume each $v_n$ is the geodesic $t \mapsto (u(t\sin\theta_n)\, ,\, v(t\cos\theta_n))$ in $u \times v$ for some $\theta_n \in (0, \pi/2)$, and $\theta_n \to 0$.
Note that the normalized projection of $v_n$ is always $u$.
Thus by \lemref{pedro 1.2}, we have that
\begin{equation}\label{pedro eq1}
P_v \, \cap \, P_{v_n} \, = \, P^{C_v}_u \times v.
\end{equation}
Observe that the right-hand side does not depend on $n$.

Let $x \in C_v$.
We will prove that $x \in P^{C_v}_u$.
By hypothesis, there exist $x_n \in C_{v_n}$ such that $x_n \to x$.
Let $w_n$ be the geodesic parallel to $v_n$ with $w_n(0) = x_n$.
Since $v_n \subset P_v$ and $P_v$ is convex, by \lemref{pedro 2.1} we can project $w_n$ onto $P_v$ to obtain a geodesic $w'_n$ in $P_v$ which is parallel
to $w_n$ and such that
\begin{equation*}
\dX(w_n,w'_n) \, = \, \dX(w_n,P_v) \, \leq \, \dX(x_n,x).
\end{equation*}
Thus
$\dX(x,P_v \cap P_{v_n})
\leq \dX(x,w_n) + \dX(w_n,w'_n)
\leq 2 \dX(x,x_n) \to 0$.
This, together with \eqref{pedro eq1} and the fact that $x \in C_v$, implies that $\dX(x, P^{C_v}_u) = 0$.
Since $P^{C_v}_u$ is closed, we see that $C_v \subseteq P^{C_v}_u$.
The reverse inclusion is obvious.
\end{proof}

\begin{corollary}				\label{old decomposition lemma}
Let $v \in \mathcal{A}$.
Assume $C_v$ contains a geodesic $u$ with $u(0) = v(0)$.
Then $C_v = P^{C_v}_u$.
Hence we can write $C_v = E \times u$ for some proper CAT(0) space $E$.
\end{corollary}

\subsection{Approachable cross sections are Euclidean}

\begin{theorem}					\label{pedro 4.1}
Let $X$ be a proper, geodesically complete $\CAT(0)$ space that satisfies the duality condition.
There is a nonnegative integer $k$ such that the cross section $C_v$ of every $v \in \mathcal{A}$ is a $k$-flat.
\end{theorem}

\begin{proof}
We prove the theorem in three steps.
We start by proving the theorem for $v \in \mathcal{A}$ such that $v(\infty) \in b\mathcal{U}$, but allow $k$ to depend on $v$.
We then remove dependence of $k$ on $v$, for $v \in \mathcal{U}$ with $v(\infty) \in b\mathcal{U}$.
We finally extend to all $v \in \mathcal{A}$.

\step
We first prove $C_v$ is flat for all $v \in \mathcal{A}$ such that $v(\infty) \in b\mathcal{U}$.
Let $p \in b\mathcal{U}$ and $v \in \mathcal{U}_p$.
By \cite[Theorem 6.15(6)]{bridson}, $C_v$ admits a canonical product splitting $C_v = Y \times H$, where $H$ is a Hilbert space and $Y$ does not admit nontrivial Clifford translations; furthermore, every isometry of $C_v$ preserves the product splitting.
By \thmref{isometric transitivity}, we see that $\Isom(Y)$ acts transitively on $Y$.
Thus $Y$ is either a single point or is unbounded; we claim the former case holds.

For suppose $Y$ is unbounded.
Since $\Isom(Y)$ is transitive, $Y$ is cocompact.
It follows from \cite{go} that every point $q \in \bd Y$ can be joined to some $q' \in \bd Y$ by a geodesic in $Y$.
In particular, $Y$ contains a geodesic.
By transitivity, there is a geodesic $u$ in $Y$ such that $u(0) = v(0)$ (we may, of course, assume $v(0) \in Y)$.
By \lemref{old decomposition lemma}, we therefore have nontrivial Clifford translations on $Y$, a contradiction.

Thus $C_v = H$.
Notice that $H$ must be finite dimensional because $X$ is proper.
Thus every $v \in \mathcal{A}$ with $v(\infty) \in b\mathcal{U}$ is isometric to some Euclidean space $\R^k$.

\step
The dimension of $C_v$ does not depend on $v \in \mathcal{U}$:
Let $p,q \in b\mathcal{U}$, $v \in \mathcal{U}_p$, and $w \in \mathcal{U}_q$.
Let $k = \dim(C_v)$ and $m = \dim(C_w)$.
Since $w$ is recurrent, there exist $\gamma_n \in \Isom(X)$ and $t_n \to +\infty$ such that $\gamma_n g^{t_n} (w) \to w$.
By \lemref{recurrent suspensions},
we see that $(\gamma_{n} v(-\infty))$ accumulates on some $q'$ in $\bd \PPar_w$.
Now $q'$ lies in the ideal boundary of the $m$-flat $\PPar_w$, so there is some geodesic $u$ in $\PPar_w$ such that $u(-\infty) = q'$.
Since $v(-\infty)$ and $v(\infty)$ are $\Isom(X)$-dual by the duality condition hypothesis, and those points of $\bd X$ which are $\Isom(X)$-dual to $v(\infty)$ form a closed $\Isom(X)$-invariant set in $\bd X$ \cite[Lemma~ 1.2]{bb}, we see that $v(\infty)$ is $\Isom(X)$-dual to $u(-\infty)$.
Thus there is an isometric embedding $(\Par_v, v) \into (\Par_u, u)$ by \lemref{Eberlein plus approachable}.
In particular, $w$ lies in a $k$-flat, and therefore $m \ge k$.
A symmetric argument shows $k \ge m$. 
\step
We complete the proof.
Let $v \in \mathcal{A}$ and write $p = v(\infty)$.
Since $b\mathcal{U}$ is dense in $\bd X$, there is a sequence of $p_n \in b\mathcal{U}$ such that $p_n \to p$ in $\bd X$.
By \lemref{forward-endpoint map is open}, we may find a sequence $(v_n)$ in $GX$ such that $v_n(\infty) = p_n$ and $v_n \to v$.
By \corref{consistent cross sections} and the previous two steps, each $\CS_{v_n}$ is a $k$-flat, for some fixed $k$.
Since $v$ is completely approachable, $\CS_v$ is a $k$-flat.
\end{proof}

Write $\rank(X)$ for $\dim(\Par_v) = 1 + \dim(\CS_v)$ of some (any) $v \in \mathcal{A}$.
Thus the parallel set $\Par_v$ of every $v \in \mathcal{A}$ is a flat of dimension $\rank(X)$.
In particular, the parallel set of every $w \in GX$ contains a flat of dimension $\rank(X)$ by density of the completely approachable geodesics.
Thus we have proved the Main Theorem.

We close this section with two observations about $\mathcal{A}$ which are only now clear.

\begin{corollary}
$v \in \mathcal{A}$ if and only if $\Par_v$ is a $k$-flat, where $k = \rank(X)$.
\end{corollary}

\begin{corollary}
$GX_{v(\infty)} \subset \mathcal{A}$ for all $v \in \mathcal{U}$.
I.e.~ for recurrent $v \in GX$, if $v$ is completely approachable then so is every geodesic forward asymptotic to $v$.
\end{corollary}

\begin{proof}
\lemref{recurrent embeddings}.
\end{proof}

\section{Application: a little bit of rank rigidity}

\label{sec:application}

Write $\rank(X)$ for the intrinsic rank of $X$ and $\dim(\bdT X)$ for the geometric dimension of its Tits boundary.
We now show that if $\rank(X) = 1 + \dim(\bdT X)$, then we have some rigidity.

\begin{theorem}					\label{maximal rank}
Let $X$ be a proper, geodesically complete $\CAT(0)$ space.
Assume some subgroup $\Gamma \le \Isom(X)$ satisfies the duality condition, and that $\rank(X) = 1 + \dim(\bdT X)$.
Then one of the following holds.
\begin{enumerate}
\item\label{it:mr minimal}  $\Gamma$ acts minimally on $\bd X$. \item\label{it:mr building}  $X$ is a symmetric space or Euclidean building of rank $\ge 2$.
\item\label{it:mr product}  $X$ splits as a nontrivial product.
\end{enumerate}
\end{theorem}

\begin{proof}
Assume case \itemrefstar{it:mr minimal} does not hold.
Our plan is to use Lytchak's rigidity theorem \cite[Main Theorem]{lytchak05} on the Tits boundary $\bdT X$ of $X$.
This theorem says that if $\bdT X$ is geodesically complete and contains proper closed involutive set (a set $A$ being \emph{involutive} meaning for every $p \in A$ and $q \in \bdT X$ with $\angle(p,q) = \pi$, we have $q \in A$), then $\bdT X$ is a spherical building or join.

So we first show $\bdT X$ is geodesically complete.
Since $\rank(X) = 1 + \dim(\bdT X)$, the Tits boundary $\bdT X$ of $X$ is covered by Euclidean unit spheres of dimension $\dim(\bdT X)$.
Thus $\bdT X$ is geodesically complete (by applying \cite[Lemma~ 3.1]{bl-centers} to the link of each point).

Next we find a proper closed involutive subset of $\bdT X$.
Since $\Gamma$ satisfies the duality condition, the orbit-closure $\cl{\Gamma p}$ in $\bd X$ of every point $p \in \bd X$ is a minimal nonempty closed invariant subset of $\bd X$ \cite[Proposition~ III.1.9]{ballmann}; these minimal sets are all pairwise disjoint.
(Note that by \emph{closed} we mean here closed in the cone topology.
But by lower semicontinuity of the Tits metric, they are then also closed under the Tits metric.)
So fix an arbitrary $v \in GX$, and consider the minimal sets $M = \cl{\Gamma v(-\infty)}$ and $N = \cl{\Gamma v(\infty)}$ in $\bd X$.
Now the set $M \cup N$ is clearly closed in $\bdT X$.
In \cite{ricks-onedim} (Lemma 27 and first remark following), it is shown that $M \cup N$ is proper and involutive, assuming $\Gamma$ is discrete.
However, the same arguments apply without that assumption, by simply passing to subsequences instead of using ultrafilters, so we conclude that $M \cup N$ is a proper closed and involutive subset of $\bdT X$.

Thus Lytchak's rigidity theorem \cite[Main Theorem]{lytchak05} applies, and we conclude that $\bdT X$ is a spherical join or building of dimension at least $1$.
By Leeb's theorem \cite[Main Theorem]{leeb}, either case \itemrefstar{it:mr building} or \itemrefstar{it:mr product} holds.
\end{proof}

\begin{remarks}
(1)
If $\Isom(X)$ acts cocompactly on $X$, then $1 + \dim(\bdT X)$ coincides with the dimension of a maximal flat in $X$ by Kleiner \cite[Theorem C]{kleiner}.
Thus in this case, the condition $\rank(X) = 1 + \dim(\bdT X)$ is equivalent to the condition $\rank(X) = \max \setp{\dim F}{F \text{ is a flat in } X}$.

(2)
Let $X$ be a proper, geodesically complete $\CAT(0)$ space satisfying the duality condition.
By Ballmann \cite[Theorem III.2.3]{ballmann}, case (i) is equivalent to the geodesic flow on $GX$ having a dense orbit mod $\Gamma$.
\end{remarks}

Using the deRham decomposition of $X$ (which exists and is unique by Foertsch and Lytchak \cite[Theorem 1.1]{fl}), we can state the following corollary.

\begin{corollary}				\label{maximal rank corollary}
Let $X$ be a proper, geodesically complete $\CAT(0)$ space that satisfies the duality condition.
Let $X = X_1 \times \dotsb \times X_n$ be the maximal de Rham decomposition of $X$, so that each $X_i$ is neither compact nor a product.
Assume $\rank(X) = 1 + \dim(\bdT X)$.
Then for each de Rham factor $X_i$ of $X$, either
\begin{enumerate}
\item\label{it:mrc minimal}  $\Isom(X_i)$ acts minimally on $\bd X_i$ and the geodesic flow on $GX_i$ has a dense orbit mod $\Isom(X_i)$, or
\item\label{it:mrc building}  $X_i$ is a symmetric space or Euclidean building of rank at least two.
\end{enumerate}
\end{corollary}

\begin{proof}
Each $X_i$ satisfies the hypotheses of \thmref{maximal rank} with $\Gamma_i = \Isom(X_i)$, but none splits as a nontrivial product by hypothesis.
\end{proof}

\begin{remark}
Let $X = X_1 \times \dotsb \times X_n$ be as in \corref{maximal rank corollary}, except that $\rank(X) \neq 1 + \dim(\bdT X)$.
If the $\CAT(0)$ Rank Rigidity Conjecture holds, then at least one of the de Rham factors $X_i$ must admit a rank one axis. \end{remark}

\appendix

\section{A surface example}

\begin{theorem}						\label{surface theorem}
Let $S$ be an orientable closed surface of genus $>1$, and let $g = g_0$ be a
$C^\infty$ nonpositively curved Riemannian metric on $S$.
Then there is a sequence $g_n$ of $C^\infty$ nonpositively curved Riemannian metrics on $S$ such that $g_n$ $C^0$-converges to a $C^0$ nonpositively curved Riemannian metric $g_\infty$ on $S$ under which every closed $g_\infty$-geodesic is contained in an isometrically immersed flat cylinder.
\end{theorem}

\begin{remarks}
(1)
The $C^0$ metric $g_\infty$ induces a geodesic metric on $S$.
In this case ``nonpositively curved'' means locally $\CAT(0)$.
(See Theorem 4.11 in \cite{burtscher}.
Here Burstscher proves $C^0$ Riemannian manifolds are length spaces.
But compact length spaces are geodesic spaces; see \cite[p.35]{bridson} for this.)

(2)
With more care one can possibly arrange to have a $C^\infty$ path $g_t$ of such metrics $C^0$-converging to $g_\infty$.
\end{remarks}

\begin{proof}
Let $S$ and $g$ as in the Theorem, and give $S$ an orientation.
Let $\gamma:X\ra S$ be a nontrivial geodesic in $(S,g)$, where $X$ is either the circle $\bS^1$, or the interval $[0,1]$.
Also, we will denote by $X(\ell)$ the circle of length $\ell$ or the interval $[0,\ell]$.
For each $u\in X$ let $V(u)\in T_{\gamma(u)}S$ be the unit vector perpendicular to $\gamma'(u)$ and such that $(\gamma'(u), V(u))$ is positively oriented.
We get a map $E=E_{\gamma,g}:X\times\R\ra S$ given by $E(u,s)=$ exp$_{\gamma(u)}(sV(u))$.
Then $E$ is an immersion near $\gamma$, that is, there is $\epsilon>0$ such that $E=E_{\gamma,g}$ restricted to $X\times [-2\epsilon,2\epsilon]$ is an immersion.
The supremum of all such $\epsilon$ will be denoted by $\epsilon_{\gamma,g}$.
Therefore, for all $\epsilon<\epsilon_{\gamma,g}$, the map $E$ is an immersion on $X\times [-2\epsilon,2\epsilon]$.
In this case the pullback of $g$ to $X\times [-2\epsilon,2\epsilon]$ is a Riemannian metric, which we denote by $g_\gamma$.
If $X$ is an interval we will assume $\gamma$ extends to a larger interval $X'$; if $X=\bS^1$ then $X'=X$.

For two Riemannian metrics $g_1,g_2$ we write $g_1\leq g_2$ if $g_1(x,x)\leq g_2(x,x)$ for all tangent vectors $x$.
An \defn{arc} in $X$ is a subspace $A\sbs X$ homeomorphic to a closed interval.
The set $\gamma (A)$ will also be called an arc.
We will call the set $E(A\times [-\epsilon,\epsilon])$ the \defn{$\epsilon$-rectangle of $A$}.
We will need the following lemma.

\begin{lemma}						\label{surface sublemma 1}
Let $S$, $\gamma$, $g$ and $E$ as above.
Let $\epsilon<\epsilon_{\gamma,g}/2$, and $\delta>0$.
Then there is a $C^\infty$ nonpositively curved Riemannian metric $g_1$ on $S$ such that
\begin{enumerate}
\item[(1)]
$(1-\delta) g_0 \leq g_1\leq (1+ \delta) g_{0}$.
\item[(2)]
$g_1=g$ outside $E(X'\times [-\epsilon/2,\epsilon/2])$.
\item[(3)]
The curve $\gamma$ is a $g_1$-geodesic.
\item[(4)]
There is $\epsilon'\in(0,\epsilon/4)$ such that $X\times [-\epsilon',\epsilon']$ with metric $(g_{1})_{\gamma}$ is isometric to the (flat) Euclidean product of $X(\ell)$ with $[-\epsilon',\epsilon']$, where $\ell$ is the $g_1$-length of $X$.
\item[(5)]
Let $A$ be an arc in the interior of $X$
and assume that the curvature $K_g$ is zero on an open set containing $E(A\times [-\epsilon,\epsilon])$.
Then we can arrange that $g_\gamma=(g_{1})_\gamma$ on $A\times [-\epsilon,\epsilon]$.
\end{enumerate}
\end{lemma}

Postponing the proof of \lemref{surface sublemma 1} for the moment, we proceed with our proof of \thmref{surface theorem}.
Enumerate the free homotopy classes of loops in $S$:  $C_1, C_2,C_3,...$.
Write $\delta_n=\frac{1}{2^{n+1}+2}$.
As our first step just choose a closed $g$-geodesic $\gamma_1$ in $C_1$, and apply \lemref{surface sublemma 1} with $X=\bS^1$, $\gamma=\gamma_1$, $\delta=\delta_1$, and any $\epsilon<\epsilon_{\gamma_1,g}/2$ to obtain a metric $g_1$.
Write $\epsilon'=2\eta_1>\eta_1$, where $\epsilon'$ is as in \lemref{surface sublemma 1}.
That is, $\gamma_1$ is contained in an isometrically immersed flat cylinder $\mathcal{C}_1(2\eta_1)$ of width $2\eta_1>\eta_1$, with respect to the metric $g_1$.
(We will denote the image of $\mathcal{C}_1$ also by $\mathcal{C}_1$).
The next step is to choose a closed $g_1$-geodesic $\gamma_2$ in $C_2$, and apply \lemref{surface sublemma 1} with $X=\bS^1$, $\gamma=\gamma_2$, $\delta=\delta_2$, and $\epsilon=\epsilon_2$ small (how small will be determined below) to obtain a metric $g_2$.
Write $E_2=E_{\gamma_2,g_2}$ and $E_2(\epsilon_2)= E_2(\bS^1\times[-\epsilon_2,\epsilon_2])$.
The geodesic $\gamma_2$ is contained in a flat immersed cylinder $\mathcal{C}_2(2\eta_2)$ of some width $2\eta_2>\eta_2$.
But this step may change the cylinder of step 1.
This is an unavoidable problem, but we can minimize the problem: because of (5) of \lemref{surface sublemma 1}, and the fact that the width of the cylinder in step 1 is \defn{strictly larger} than $\eta_1$, we can choose $\epsilon_2$ so small that the new $g_2$-width of the cylinder of $\gamma_1$ is still $>\eta_1$ (even though the width decreases a bit).
Here is a more detailed description of how to do this.
Let $\mathcal{C}_1(\frac{3}{2}\eta_1)\sbs\mathcal{C}(2\eta_1)$ be the flat isometrically immersed cylinder of width $\frac{3}{2}\eta_1$.
The intersection of the image of $\gamma_2$ with the cylinder $\mathcal{C}_1(\frac{3}{2}\eta_1)$ of step 1 is a finite set of arcs $\gamma_2(A_i)$.
Choose $\epsilon_2$ small enough so that the $\epsilon_2$-rectangles of the $A_i$ are contained in the interior of $\mathcal{C}_1(2\eta_1)$ and the set
\[\bigg(\mathcal{C}_1(2\eta_1)\,\cap\,E_2(\epsilon_2)\bigg)
\,\setminus\,
\bigcup_i\epsilon_2{\mbox{-rectangle of}} \,\,A_i\]
is outside the $\eta_1$ neighborhood of $\gamma_1$.
Since the curvature is zero on $\mathcal{C}_1(2\eta_1)$, (5) of \lemref{surface sublemma 1} implies that we can arrange for the metrics $g_1$ and $g_2$ to coincide on the $\epsilon_1$-rectangles of the $A_i$.
In this way, after step 2, $\gamma_1$ is still contained in a flat isometrically immersed cylinder of width $>\eta_1$.

Now, proceed inductively to obtain $g_n$ and $\gamma_n$ contained in an isometrically immersed flat cylinder of width $>\eta_n$.
For the $n+1$ step we procceed similarly, choosing $\epsilon_{n+1}$ so small that all $\gamma_i$, $i\leq n$, are still contained in isometrically immersed flat cylinders of width $>\eta_i$.
In this way we define $g_n$ for $n=1,2,3,\dots$.
Next we prove convergence.

Recall $\delta_n=\frac{1}{2^{n+1}+2}$, so $1-\delta_n=\frac{2^{n+1}+1}{2^{n+1}+2}$ and $1+\delta_n=\frac{2^{n+1}+3}{2^{n+1}+2}\leq \frac{2^{n+1}+2}{2^{n+1}+1} =\frac{1}{1-\delta_n}$.
Now, from (1) of \lemref{surface sublemma 1} we have $ (1-\delta_n)g_n \leq g_{n+1}\leq (1+\delta_n) g_n \leq \frac{1}{1-\delta_n}g_n$.
Hence $a_n g\leq g_{n+1}\leq \frac{1}{a_n} g$, where $a_n=\prod_{i=1}^n (1-\delta_i)$.
One can show, by induction, that
$a_n=\frac{2^{n+3}+2}{2^{n+4}}\geq \frac{1}{2}$.
Hence $\frac{1}{2} g\leq g_{n+1}\leq 2 g$.

Let $x$ be a tangent vector.
Then (1) of \lemref{surface sublemma 1} implies
\[-\delta_n g_n(x,x)\leq g_{n+1}(x,x)-g_n(x,x)\leq \delta_ng_n(x,x).\]
Therefore
\[|g_{n+1}(x,x)-g_n(x,x)|\leq \delta_ng_n(x,x)\leq 2\delta_n g(x,x)
= \frac{2}{2+2^{n+1}}g(x,x)\leq \frac{1}{2^n}g(x,x).\]
Replacing $x$ by $x+y$ and using the triangular inequality we obtain
$|g_{n+1}(x,y)-g_{n}(x,y)|\leq \frac{1}{2^{n-1}}(g(x,x)+g(y,y))$.
Therefore, for each pair $x,y$ the sequence $g_n(x,y)$ is a Cauchy sequence, hence converges.
This gives a $C^0$-symmetric bilinear form $g_\infty$ on $TS$.
We certainly have $g_\infty(x,x)\geq 0$.
But we have showed that $\frac{1}{2}g\leq g_{n+1}$.
This shows $\frac{1}{2}g(x,x) \leq g_\infty(x,x)$.
Therefore $g_\infty$ is nondegenerate.
This proves the theorem.
\end{proof}

\begin{proof} [Proof of \lemref{surface sublemma 1}]
We assume $\gamma$ has speed 1.
Let $\ell$ be the length of $\gamma$ (recall $X$ is an interval or a circle).
For simplicity we change a bit the domains of $\gamma$ and $E$:  We replace $X$ by $X(\ell)$.
We use coordinates $(u,v)\in X(\ell)\times [-\epsilon,\epsilon]$.
The velocity vectors of the $u$-lines $u\mapsto (u,v_0)$ and $v$-lines $v\mapsto (u_0,v)$ will be denoted by $\p_u$ and $\p_v$, respectively.
Recall that $g_\gamma$ is the pullback of $g$ by the immersion $E$; we consider $X(\ell)\times [-\epsilon,\epsilon]$ with this metric.
Note that the $v$-lines are speed one geoesics, and the $g$-geodesic $\gamma$ corresponds to the $u$-line $u\mapsto (u,0)$.
We have $g_\gamma(\p_v,\p_v)=1$ and $g_\gamma(\p_u,\p_v)=0$.
Write $g_\gamma(\p_u,\p_u)=f^2>0$.
Note that $f(u,0)=1$, for all $u$.
Hence the metric $g_\gamma$ on $X(\ell)\times [-\epsilon,\epsilon]$ can be writen as $ f^2(u,v)du^2+dv^2$.
The curvature of this metric is $-\frac{f_{vv}}{f}$.
Hence $f_{vv}\geq 0$.
Also, since $u\mapsto (u,0)$ is a geodesic, one can deduce from the equations of a geodesic that $f_v(u,0)=0$, for all $u$.
Hence $v\mapsto f(u,v)$ has a minimum at $v=0$, and $f(u,v)\geq 1$, for all $(u,v)$.

To construct the metric $g_1$ we will need the following functions.
For $t\in [0,1]$, let $\rho_t:\R\ra\R$ be $C^\infty$ and such that
(1) $\rho_t(-z)=-\rho_t(z)$,
(2) $\rho_t(z)=z-2t$ for all $z\geq 3$,
(3) $\rho''_t(z)\geq 0$ for all $z\geq 0$,
(4) $|z-\rho_t(z)|\leq 2t$.
Note that property (4) for $t=0$ implies $\rho_0=1_\R$.
We also demand
(5) $\rho_1(z)=0$ whenever $|z|\leq 1$.
For $\eta>0$ define
$\rho_{\eta,t}(z)= \eta\rho_t(\frac{z}{\eta})$.
We write $\rho_\eta=\rho_{\eta,1}$.

We will also need the following functions.
For $\eta>0$ small and $t\in [0,1]$, let $\sigma_{\eta,t}:\R\ra\R$ be $C^\infty$ such that
(1) $\sigma_{\eta,t}(z)=z$, for $|z|\leq \eta$,
(2) $\sigma_{\eta,t}(-z)= -\sigma_{\eta,t}(z)$,
(3) $\sigma_{\eta,t} (z)=z+2\eta t$, for $z\geq \sqrt{\eta}$,
(4) $1\leq \frac{d}{dz}\sigma_{\eta,t}(z)\leq 1+3t\sqrt{\eta}$,
(5) $|\sigma_{\eta,t} (z) - z|\leq 2\eta t$,
(6) $|\frac{d}{dt} \sigma_{\eta,t}(z)| \leq 3\eta$, for all $z$ and $t\in [0,1]$.
Note that it follows that $\sigma_{\eta,0}=1_\R$.
We will write $\sigma_\eta=\sigma_{\eta,1}$.

We assume $0<\epsilon <\frac{1}{9}$, and $\eta>0$ with $\sqrt{\eta}\leq\epsilon/4$.
Define the diffeomorphism $S_{\eta}:X\times\R\ra X\times\R$ by
$S_{\eta}(u,v)=(u,\sigma_{\eta}(v))$.
On $X\times [-\epsilon-2\eta,\epsilon+2\eta]$ define the metric
$h_{\eta}=f^2(u,\rho_{\eta}(v))du^2+dv^2$.
Finally, on $X\times [-\epsilon,\epsilon]$ define the metric
$g_{\eta}=S_{\eta}^\ast h_{\eta}$, that is
\begin{equation}
g_{\eta}(u,v)=f^2\big(u,\rho_{\eta}(\sigma_{\eta}(v))\big)
du^2\,\,+\,\,\big( \sigma'_{\eta}(v)\big)^2dv^2\tag{$\ast$}
\end{equation}
We have the following properties.
\begin{enumerate}
\item[(a)]
For $|v|\geq \sqrt{\eta}$ and all $u$ we have $g_\gamma(u,v)=g_{\eta}(u,v)$.
\item[(b)]
$\frac{\p^2}{\p v^2}f(u,\rho_{\eta}(v))\geq 0$, hence $h_{\eta}$ is nonpositively curved.
Therefore $g_{\eta}$ is nonpositively curved.
\item[(c)]
$(1-C\sqrt{\eta})g_\gamma\leq g_{\eta}\leq (1+C\sqrt{\eta})g_\gamma$, for some constant $C$.
\item[(d)]
On $X\times [-\eta,\eta]$ we have $h_{\eta}=du^2+dv^2$, hence $X\times [-\eta,\eta]$ with metric $h_{\eta}$ is isometric to a flat cylinder.
Since $S_{\eta}$ sends $X\times [-\eta,\eta]$ to itself, the same is true for $X\times [-\eta,\eta]$ with metric $g_{\eta}$.
\end{enumerate}

\noindent
Properties (a), (b) and (d) follow directly from the definitions.
We prove (c).
From the definitions we have
\[|\rho_\eta(\sigma_{\eta}(v))-v|
\leq |\rho_\eta(\sigma_{\eta}(v))-\sigma_{\eta}(v)| + |\sigma_{\eta}(v)-v|
\leq 2\eta+2 \eta=4\eta.\]
Let $C_1$ be the $C^1$-norm of $f$.
Then
$|f(u,\rho_\eta(\sigma_{\eta}(v))-f(u,v)|\leq 4\eta C_1$.
Therefore
$|f^2(u,\rho_\eta(\sigma_{\eta}(v))-f^2(u,v)|\leq 4\eta C_1(C_1+C_1)
= 8C_1^2\eta$.
We write $C=8C_1^2+9$.
Let $x=a\p_u+b\p_v$ be a tangent vector.
Then
\begin{align*}
\bigg|g_\eta(u,v)(x,x)-g_\gamma(u,v)(x,x)\bigg|
&=
\bigg|\big(f^2(u,\rho_\eta(\sigma_{\eta}(v))a^2+(\sigma'_{\eta}(v))^2b^2\big)
- \big(f^2(u,v) a^2+b^2\big)\bigg| \\
&\leq \bigg|f^2(u,\rho_\eta(\sigma_{\eta}(v)))-f(u,v)\bigg|a^2
\,\,+\,\,
\bigg| (\sigma'_{\eta}(v))^2-1) \bigg| b^2 \\
&\leq C\eta a^2 \,\,+\,\,9\sqrt{\eta} \,b^2 \\
&\leq C\sqrt{\eta} \big( f^2(u,v)a^2+b^2\big)
\,\,=\,\,
C\sqrt{\eta}\,g_\gamma(u,v)(x,x).
\end{align*}

\noindent
In the last inequality we are using $\eta<1$, $f\geq 1$, $C\geq 9$.
Also, in the second inequality we are using (4) of the definition of $\sigma$.
This proves (c).

We prove one more property of the metric $g_\eta$ on $X\times [-\epsilon,\epsilon]$ given by ($\ast$).

\begin{lemma}						\label{surface sublemma 2}
Let $A$ be an arc in the interior of $X$, and assume the curvature is zero on an open set containing the $\epsilon$-rectangle of $A$.
Then we can modify $g_\eta$ so that $g_\eta=g_\gamma$ on $A\times [-\epsilon,\epsilon]$.
\end{lemma}

\begin{proof}
Let $A=[a,b]$ be an arc.
Then there is $\chi>0$ such that on
$U=(a-\chi,b+\chi) \times [-\epsilon,\epsilon]$, the curvature of $g_\gamma$ is zero.
We have to prove that we can modify $g_\eta$ so that $g_\eta=g_\gamma$ on $A\times [-\epsilon,\epsilon]$.
Since the curvature is zero we have $f_{vv}=0$.
But we also have $f_v(u,0)=0$ and $f(u,0)=1$.
Therefore $f\equiv 1$ on $U$, hence $g=du^2+dv^2$ on $U$.
On the other hand, from the definitions, one can see that on $U$ we have $h_\eta=du^2+dv^2$ and $g_\eta=du^2+(\sigma'_{\eta})^2 dv^2$, which is isometric to $h_\eta=du^2+dv^2$ via $S_\eta$.
We now change $S_\eta$.
Let $\theta:X\ra [0,1]$ such that $\theta \equiv 1$ outside $(a-\chi/2,b+\chi/2)$ and $\theta \equiv 0$ on a neighbohood of $A$.
Define $\barS(u,v)=(u,\sigma_{\eta,\theta(u)}(v))$.
Also define the modified metric $\bg_\eta=\barS^\ast h_\eta$.
It can now be shown from the definitions that properties (a), (b) and (d) still hold for $\bg_\eta$; moreover we also have $\bg_\eta= du^2+dv^2=g_\gamma$ on $A\times [-\epsilon,\epsilon]$, as required.
One may have now a new problem with property (c) since there is a new term $\frac{d}{du}\sigma_{\eta,\theta(u)}(v)$ in the derivative of $\barS_\eta$ that could be large.
To solve this note that property (6) in the definition of $\sigma$ implies
$|\frac{d}{du}\sigma_{\eta,\theta(u)}(v)|
= |\theta'(u)|\,| \frac{d}{dt}\sigma_{\eta,t}(v)|_{t=\theta(u)}|
\leq 3\eta |\theta'(u)|$.
Hence, we can just fix $\chi$ and $\theta$ and take $\eta$ very small.
In this way it is straightforward to show that (c) still holds, maybe with a larger $C$ which depends on the fixed number $\chi$ and fixed function $\theta$.
\end{proof}

We divide the remainder of the proof in three cases.\\

\noindent{\bf 1. $X=\bS^1$ and $E$ is an Embedding.}

\noindent
If $E$ is an embedding we can define the metric $g_1$ by demanding $g_1=g$ outside the image of $E$ and equal to $E_\ast g_\eta$ inside the image of $E$, where $g_\eta$ is as in equation ($\ast$).
By property (a) this metric is well defined.
By choosing $\eta$ small we get that this $g_1$ satisfies properties (1)-(4) of \lemref{surface sublemma 1}.
Property (5) follows from \lemref{surface sublemma 2}.\\

\noindent{\bf 2. $X=[0,\ell]$ is an Interval and $E$ an Embedding.}

\noindent
First we have to extend the domain of $E$.
Since $\epsilon< \epsilon_{\gamma,g}/2$ there is $\chi>0$ such that $E$ extends to an embedding $E:I\times [-\epsilon,\epsilon]\ra S$, where $I=[-\chi,\ell+\chi]$.
Let $\theta:I\ra[0,1]$ be smooth and such that $\theta\equiv 1$ on a neighborhood of $[0,\ell]$ and $\theta\equiv 0$ near the end points of $I$.
Now consider the following extensions of $h_\eta$ and $S_\eta$.
Define $h_\eta(u,v)=f^2(u,\rho_{\eta,\theta(u)}(v))du^2+dv^2$, and $S_\eta(u,v)=(u,\sigma_{\eta,\theta(u)}(v))$.
Finally define $g_\eta=S_\eta^\ast h_\eta$.
It can be directly checked from the definitions that properties (a), (b) and (d) above still hold.
Also, from the definition of the newly extended $g_\eta$ we have that $g_\eta=g_\gamma$ near $\{-\chi\}\times [-\epsilon,\epsilon]$ and $\{\ell+\chi\}\times [-\epsilon,\epsilon]$.
Property (c) can be proven as in the proof of \lemref{surface sublemma 2}: Fix $\chi$ and takes $\eta$ sufficiently small.
Define $g_1=g$ outside the image of $E$ and $g_1=E_\ast g_\eta$ on the image of $E$.
By (a)-(d) $g_1$ is well defined and satisfies (1)-(4) of \lemref{surface sublemma 1}.
In fact a bit more than (3) holds:  $\gamma:I\ra S$ is still a geodesic.
Also, (5) follows from \lemref{surface sublemma 2}.\\

\noindent{\bf 3. General Case.}

\noindent
We assume $X=\bS^1(\ell)$.
The case of $X$ being an interval is similar.
We now allow the closed geodesic $\gamma$ to have self-intersections.
To simplify our argument we assume $\gamma$ has exactly one self-intersection at the point
$p=\gamma(u_1)=\gamma(u_2)$, $u_1\neq u_2$; the case with more self-intersections is similar.
Let $A$ be an arc in $\bS^1(\ell)$ containing $u_1$ as middle point, and $\epsilon$ small such that $E$ restricted to $A\times [-\epsilon,\epsilon]$ is an embedding.
By case 2, and taking $\epsilon$ even smaller if necessary, we can assume (1) the curvature on the $\epsilon'$-rectangle of $A$ is zero, where $\epsilon'<\epsilon$, (2) $\gamma$ still has exactly one self-intersection at some $q=\gamma(u'_1)=\gamma( u'_2)$ near $p$, with $u'_1\in A$, (3) the intersection of the $\epsilon'$-rectangle of $A$ with the image of $\gamma$ is exactly two arcs $\gamma(A)$ and $\gamma(A')$, where $A\cap A'=\emptyset$, $u'_2\in A'$, and $\gamma(A')\cap E(\p A\times [-\epsilon,\epsilon])=\emptyset$.
After applying case 2 the length of $\gamma$ may change a bit, but we will still denote it by $\ell$.

\begin{remark}
Note that after applying the (already-proved) Case 2 of \lemref{surface sublemma 1}, $\gamma$ may change a bit, but the new $\gamma$ can be chosen as close as the old $\gamma$ by taking $\delta$ in \lemref{surface sublemma 1} (Case 2) as small as needed.
\end{remark}

Let $A''\sbs A'$ such that
$\gamma(A'')= \gamma(A')\cap \gamma(A\times [-\epsilon'/4,\epsilon'/4])$.
Note that $u'_2\in A''$.
Now, let $\epsilon''>0$ be small so that (1) the $\epsilon''$-rectangle of $A''$ is contained in the interior of the $\epsilon'$-rectangle of $A$, (2) the intersection of the $\epsilon''$-rectangles of $A$ and $A''$ is disjoint from $E(\p A\times [-\epsilon'',\epsilon''])$ and $E(\p A''\times [-\epsilon'',\epsilon''])$.
Let $g_\eta$ be the metric on $X\times [-\epsilon,\epsilon]$ given in equation ($\ast$), with the new $\epsilon=\epsilon''$ and $X=\bS^1(\ell)$.
By \lemref{surface sublemma 2} we can assume that $g_\eta=g_\gamma$ on $A\times [-\epsilon'',\epsilon'']$ and on $A''\times [-\epsilon'',\epsilon'']$.
As before we define $g_1=g$ outside the image of $E$ and $g_1=E_\ast g_\eta$ on the image of $E$.
Note that this metric is well defined because $E_\ast g_\eta=g$ on the $\epsilon''$-rectangles of $A$ and $A''$.
As in the previous cases, $g_1$ satisfies (1)-(5) in the statement of \lemref{surface sublemma 1}.
\end{proof}

\bibliographystyle{amsplain}
\bibliography{refs}

\end{document}

%% file: artcfg.tex
\usepackage{amssymb}
\usepackage{hyperref}
\usepackage{tikz}

\usepackage{environ}

\usepackage{import}

\theoremstyle{plain}
\newtheorem{thm}{Theorem}%[section]
\newtheorem{theorem}[thm]{Theorem}
\newtheorem*{theorem*}{Theorem}

\newtheorem{corollary}[thm]{Corollary}
\newtheorem{lemma}[thm]{Lemma}
\newtheorem*{lemma*}{Lemma}

\newtheorem*{conjecture*}{Conjecture}

\newtheorem*{fact*}{Fact}
\newtheorem*{question*}{Question}

\newtheorem*{maintheorem}{Main Theorem}
\newtheorem*{maincorollary}{Main Corollary}
\newtheorem{introtheorem}{Theorem}

\newtheorem*{crrconjecture}{CAT(0) Rank Rigidity Conjecture}

\theoremstyle{remark}

\newtheorem*{remark}{Remark}
\newtheorem*{remarks}{Remarks}

\theoremstyle{definition}
\newtheorem{definition}[thm]{Definition}

\newtheorem*{standing hypothesis}{Standing Hypothesis}

\newcounter{step}
\newcommand{\step}{\textbf{\refstepcounter{step} Step \Roman{step}. }}

\newcommand{\thmref}[1]{Theorem~\ref{#1}}

\newcommand{\corref}[1]{Corollary~\ref{#1}}

\newcommand{\lemref}[1]{Lemma~\ref{#1}}

\newcommand{\itemrefstar}[1]{(\ref*{#1})}

\newcommand{\secref}[1]{Section~\ref{#1}}

\newcommand{\defn}[1]{\emph{#1}}

\graphicspath{{pictures/}}

\newcommand{\N}{\mathbb{N}}

\newcommand{\R}{\mathbb{R}}

\newcommand{\into}{\hookrightarrow}
\renewcommand{\setminus}{\smallsetminus}

\DeclareMathOperator{\rank}{rank}

\DeclareMathOperator{\id}{id}
\DeclareMathOperator{\Isom}{Isom}

\newcommand{\set}[1]{\left\{#1\right\}}
\newcommand{\setp}[2]{\left\{#1 \mid #2\right\}}
\newcommand{\abs}[1]{\left| #1 \right|}

\newcommand{\res}[1]{\vert_{#1}}

\newcommand{\cl}[1]{\overline{#1}}
\newcommand{\bd}{\partial}
\newcommand{\double}[1]{#1 \times #1}
\newcommand{\dbX}{\double{\bd X}}
\newcommand{\bdT}{\partial_{\mathrm{T}}}

\newcommand{\lmod}{\backslash}

\newcommand{\modgp}[2]{#2 \lmod #1}%{#1 \rmod #2}
\newcommand{\modG}[1]{\modgp{#1}{\Gamma}}

\renewcommand{\epsilon}{\varepsilon}